\documentclass[leqno]{article}

\usepackage[utf8x]{inputenc}
\usepackage{amssymb,amsmath,amsthm}
\usepackage{fullpage}
\usepackage{color, soul}

\newtheorem{thm}{Theorem}[section]
\newtheorem{lem}[thm]{Lemma}
\newtheorem{cor}[thm]{Corollary}

\theoremstyle{definition}
\newtheorem{defn}[thm]{Definition}

\theoremstyle{remark}
\newtheorem{rem}[thm]{Remark}

\newcommand{\RR}{\mathbb{R}}
\newcommand{\cL}{\mathcal{L}}
\newcommand{\cA}{\mathcal{A}}
\newcommand{\BA}{1 \wedge (\beta+\alpha)}

\newcommand{\Holder}{H\"{o}lder }

\numberwithin{equation}{section}

\title{$C^{1,\alpha}$ Interior Regularity for Nonlinear Nonlocal Elliptic Equations With Rough Kernels}
\author{Dennis Kriventsov\thanks{Mathematics Department, University of Texas at Austin, Austin, Texas}}

\begin{document}

\maketitle

\begin{abstract}

We prove a $C^{1,\alpha}$ interior regularity theorem for fully nonlinear uniformly elliptic integro-differential equations without assuming any regularity of the kernel. We then give some applications to linear theory and higher regularity of a special class of nonlinear operators.

\end{abstract}

\section{Introduction}

In this paper we consider a broad class of elliptic integro-differential equations, of the type first treated in \cite{CS1}. This includes the following nonlocal Isaacs equations:
\begin{equation}\label{Isaacs}
 I(u,x)=\inf_{\beta} \sup_{\alpha} \int_{\RR^{n}}\delta_{2}u(x,y)K_{\alpha,\beta}(x,y)dy
\end{equation}
where $\delta_{2}u(x,y)=u(x+y)-2u(x)+u(x-y)$ are symmetric differences. The kernels $K_{\alpha,\beta}$ are assumed to satisfy a uniform ellipticity condition, which amounts to being bounded above and below by a multiple of the fractional Laplacian:
\begin{equation}\label{kerelliptic}
\frac{(2-\sigma)\lambda}{|y|^{n+\sigma}} \leq K_{\alpha,\beta}(x,y) \leq \frac{(2-\sigma)\Lambda}{|y|^{n+\sigma}}.
\end{equation}
Operators of this form arise from stochastic control theory. 

A more general notion of ellipticity for fully nonlinear operators, developed in \cite{CS2} for the variable-coefficient case, involves the extremal operators. For a family $\cL$ of linear operators, define
\[
 M^{+}_{\cL}u(x)=\sup_{L\in \cL}Lu(x)
\]
and
\[
 M^{-}_{\cL}u(x)=\inf_{L\in \cL}Lu(x).
\]
An operator $I$ is said to be uniformly elliptic with respect to $\cL$ if 
\begin{equation}\label{elliptic}
 M^{-}_{\cL}v(x)\leq I(u+v,x)-I(u,x)\leq M^{+}_{\cL}v(x)
\end{equation}
for all $u$, $v$ for which $I$ is well-defined; in particular for functions locally $C^{2}$ and whose tails are integrable across the kernel (more on this in Section 2). These are analogous to the extremal Pucci operators in the local second-order theory.  A particularly useful family of linear operators to consider is the ones whose kernel satisfies \eqref{kerelliptic}, as this is the lightest ellipticity assumption which is known to yield \Holder regularity of solutions (see \cite{CS1}). This class will be denoted $\cL_{0}$. In this case the extremal operators admit a convenient explicit formula:
\begin{equation}\label{maximal}
 M^{+}u(x)\equiv M^{+}_{\cL_{0}}u(x)=\int \frac{\Lambda (\delta_{2}u)^{+}-\lambda(\delta_{2}u)^{-}}{|y|^{n+\sigma}}dy.
\end{equation}

The Dirichlet problem for such an operator is formulated as follows: let $g$ be, say, in $L^{\infty}(\RR^{n}\backslash B_{1})$ and $f$ be in $L^{\infty}(B_{1})$. Then $u$ solves the boundary value problem if
\begin{equation}\label{BVP}
\begin{cases} 
 I(u,x)=f(x)  & x\in B_{1}, 
 \\
 u=g & x\notin B_{1}.
\end{cases}
\end{equation}
In general, it isn't possible to find solutions $u$ sufficiently regular so that \eqref{BVP} will make sense classically. To deal with this, \cite{CS1} develops a theory of continuous viscosity solutions, analogous to the second-order case.

The regularity theory for \eqref{BVP} has seen rapid development in recent years. As all second-order elliptic equations can be realized as limits of nonlocal equations as $\sigma \rightarrow 2$, it is of particular interest to prove estimates uniform in $\sigma$.

This was accomplished in a sequence of papers by Caffarelli and Silvestre. In \cite{CS1} they prove that if $I$ is elliptic with respect to $\cL_{0}$, solutions $u$ to \eqref{BVP} lie in $C^{0,\alpha}(B_{1/2})$ uniformly for $\sigma>\sigma_{0}>0$. Here $\alpha$ is a universal constant depending only on $n, \lambda,\Lambda,$ and $\sigma_{0}$.

The second paper \cite{CS2} proves a Cordes-Nirenberg type result in the case $\sigma_{0}>1$. However, the proof requires a stronger assumption on the kernels than \eqref{kerelliptic}. Indeed, let $\cL_{1}$ be the collection of translation-invariant linear operators in $\cL_{0}$ satisfying the following additional constraint:
\begin{equation} \label{kerassumption}
 |\nabla K(y)|\leq C_{0} |y|^{-n-\sigma-1}
\end{equation}
for all $y \neq 0$. Then if $I$ is elliptic with respect to the class $\cL_{1}$, and is sufficiently close to a translation-invariant operator, it is shown  $u\in C^{1,\alpha}(B_{1/2})$ uniformly in $\sigma>\sigma_{0}$.

In \cite{CS3} classical interior regularity ($C^{\sigma+\alpha}$) is proved for translation-invariant, convex operators $I$ elliptic with respect to the class $\cL_{2}\subset \cL_{1}$, which imposes the additional constraint
\begin{equation} \label{kerassumption2}
 |D^{2} K(y)|\leq C_{0} |y|^{-n-\sigma-2}.
\end{equation}
In the limit as $\sigma \rightarrow 2$, this gives the Evans-Krylov theorem.

These extra kernel assumptions are used in an integration by parts argument to reduce the influence of the boundary data. This is an issue which is purely nonlocal in nature, and has no parallels in the second-order theory.

Our main purpose is to prove the interior $C^{1,\alpha}$ estimate while only assuming uniform ellipticity with respect to $\cL_{0}$, not $\cL_{1}$. This is accomplished in two stages. The first is an estimate on ``constant-coefficient'' equations:
\begin{thm} \label{part1intro}
 Let $\sigma>\sigma_{0}>0$, $I$ be translation-invariant and uniformly elliptic with respect to $\cL_{0}$, $u$ satisfy \eqref{BVP} in the viscosity sense with $f=0$ and $g\in C^{0,1}$. Then there is an $\alpha>0$ and $C$ depending only on $n, \lambda,\Lambda,$ and $\sigma_{0}$ such that $u\in C^{1,\alpha}(B_{1/2})$ and
 \begin{equation}
  \|u\|_{C^{1,\alpha}(B_{1/2})}\leq C\left(\|g\|_{C^{0,1}(B_{1}^{C})}+\|u\|_{L^{\infty}(\RR^{n})}\right).
 \end{equation}
\end{thm}
The second-order analogue of this is a corollary of the Krylov-Safanov theorem (see \cite[Corollary 5.7]{CC}). Theorem $13.1$ in \cite{CS1} proves a variant of this where the dependence is only on $\|g\|_{L^{\infty}}$, but an additional assumption is placed on the kernels (it would reduce to $\cL_{1}$ if made for every rescaling of $I$). In general, the arguments consist of forming incremental quotients
\[
 \frac{u(x+h)-u(x)}{|h|^{\beta}}
\]
and noticing that (because of translation invariance) they solve an elliptic equation with bounded right-hand side. Then an application of the \Holder regularity theorem shows these quotients are $C^{0,\alpha}$, so $u$ is $C^{\beta+\alpha}$. In the nonlocal case the difficult step is actually showing they are globally bounded uniformly in $h$, which in general won't be the case.

Our proof exploits the fact that it's enough for the incremental quotients to be integrable to obtain interior \Holder estimates. The method requires using scaling arguments to prove weighted \Holder estimates up to the boundary, and then using those estimates to show integrability of the incremental quotients. This proof is carried out in Section 3.

The second step is a perturbative argument that allows for arbitrary right-hand sides, variable coefficient kernels, and bounded boundary data. The conclusion is as follows (a precise statement for the variable-coefficient case, along with the proof, is given in section 4):
\begin{thm} \label{part2intro}
 Let $\sigma>\sigma_{0}>1$, $I$ be translation-invariant and uniformly elliptic with respect to $\cL_{0}$, $u$ satisfy \eqref{BVP} in the viscosity sense with $f,g\in L^{\infty}$. Then there is an $\alpha>0$ and $C$ depending only on $n, \lambda,\Lambda,$ and $\sigma_{0}$ such that $u\in C^{1,\alpha}(B_{1/2})$ and
 \begin{equation}
  \|u\|_{C^{1,\alpha}(B_{1/2})}\leq C\left(\|f\|_{L^{\infty}(B_{1})}+\|u\|_{L^{\infty}(\RR^{n})}\right).
 \end{equation}
\end{thm}
The argument is similar to the proof of Theorem 52 in \cite{CS2}. An approximation lemma is coupled with Theorem \ref{part1intro} in a sequence of rescalings. However, as Theorem \ref{part1intro} is weaker than Theorem 13.1 in \cite{CS1}, additional steps are required. In fact we repeat the iteration argument several times, each time using the conclusion of the previous step within the iteration to obtain improved regularity.

To prepare for the applications in later sections, in Section 5 we provide an explanation of a by-now standard approximation procedure for nonlocal operators. Introduced in \cite{CS3}, 
the idea is to replace kernels on a small ball around the origin by a multiple of the fractional Laplacian. The resulting equation will have classical solutions, and this can be used to justify proving a priori estimates rather than working with viscosity solutions directly.

In Section 6 we present an application to linear integro-differential equations of the following type:
\begin{equation}\label{linintro}
 Lu(x)= \int_{\RR^{n}}\frac{\delta_{2}u(x,y)a(x,y)}{|y|^{n+\sigma}}dy,
\end{equation}
where $\lambda\leq a(x,y)\leq \Lambda$. We show that if $a$ 
is \Holder in $x$ uniformly in $y$ and $f,g$ are \Holder as well, then $u$ is $C^{\sigma+\alpha}(B_{1/2})$. This is a natural generalization of the Schauder estimate to these nonlocal equations.

A more general version of this (including higher regularity) is proved in \cite{BFV}, but there additional regularity in $y$ assumptions are placed on the kernel. Our proof is based heavily on their technique. This type of estimate has also been studied with a combination of probabilistic and analytic methods in
\cite{B}. 
There the global problem ($Lu=f$ on the entire $\RR^{n}$) is shown to admit Schauder estimates for very general linear operators. A purely analytic approach is taken in \cite{DK}, where a local result is proven (see Corollary 1.7), but with dependence on the $C^{1,\alpha}$ norm of $g$. They also prove a variety of delicate results for the global problem. The novel aspect of our estimate is how it treats the dependence on the derivatives of the boundary data: rather than admit an explicit dependence (as in \cite{DK}) or attempt to integrate it away due to extra smoothness of the kernels (as in \cite{BFV}), 
we use the natural scaling of the operator to remove it.

In Section 7 this linear theory is applied to the question of higher regularity for a special class of nonlinear integro-differential operators:
\begin{equation}\label{nonlinex}
 I(u,x)=(2-\sigma)\int_{\RR^{n}}\frac{\rho({\delta_{2}u(x,y))}dy}{|y|^{n+\sigma}}.
\end{equation}
$\rho$ 
is assumed smooth and $0<\lambda<\rho'<\Lambda<\infty$ guarantees uniform ellipticity with respect to $\cL_{0}$. We prove an interior $C^{3,\alpha}$ estimate for solutions of the Dirichlet problem, assuming $f,g$ sufficiently smooth. This is an example of an operator for which ellipticity in the sense of \eqref{elliptic} is more natural than the interpretation as an Isaacs equation \eqref{Isaacs}. More importantly, it's a non-convex nonlinear operator which admits classical solutions. We feel there are a number of unanswered questions about operators of such form that could prove interesting: Are solutions to the Dirichlet problem smooth ($C^{\infty}$)? What kinds of (possibly nonlocal) second-order operators can be realized as limits of \eqref{nonlinex} as $\sigma \rightarrow 2$? What can be said about degenerate elliptic operators of this type, where the assumption $\lambda<\rho'<\Lambda$ is relaxed?

A natural question is whether the arguments in Sections 3 and 4 can be adapted to prove some analogue of the Evans-Krylov theorem in \cite{CS3}, but without assuming ellipticity with respect to $\cL_{2}$. For instance, are solutions to the Dirichlet problem for the maximal operator \eqref{maximal} classical? While this may be possible to show, it seems it requires substantial extra ideas. Notice that our method has two steps: first, a proof of an interior estimate under excessive conditions on the boundary data and second, a perturbative argument that exploits the fact that the amount of regularity assumed on the boundary data in the first step is superfluous relative to the scaling of the equation.

When approaching the Evans-Krylov theorem from this perspective, the main challenge in the first step will be to show that if $u$ satisfies the Dirichlet problem for $I=\sup_{\alpha} L_{\alpha}$ convex, then $L_{\alpha}u$ is integrable against the weight $(1+|x|)^{-n-\sigma}$. While the ideas we use here (weighted estimates up to the boundary and boundary regularity) may be helpful, it's not clear how. As for the second part, it would likely require a nonlinear Schauder theorem (see Section 8.1 in \cite{CC} for the second-order analogue), which would be an interesting result in its own right in the nonlocal context.

\section{Preliminaries}

In this section we collect some important definitions and results that will be used throughout the paper.

\subsection{Nonlocal Operators}

Let $\omega_{\sigma}$ be the following weight:
\begin{equation}\label{weight}
 \omega_{\sigma}(x)=\frac{1}{(1+|x|)^{n+\sigma}}.
\end{equation}
Associated to $\omega_{\sigma}$ is the space of functions integrable against the weight $L^{1}(\RR^{n},\omega_{\sigma})$. In general we'll be dealing with $\sigma>\sigma_{0}>0$, and in this case will use the notation $\omega \equiv \omega_{\sigma_{0}}$.

We'll use the notation $C^{1,1}(x)$ for functions $u$ for which there are quadratic polynomials $P,Q$ such that $Q(x)=u(x)=P(x)$ and $ Q\leq u\leq P$ on $B_{r}(x)$ for some small $r$. 

\begin{defn}\label{nonlocalopdef}
 A \emph{nonlocal operator} $I(u,x)$ on $B_{1}$ is a map, for each $x$, of $C^{1,1}(x)\cap L^{1}(\RR^{n},\omega_{\sigma})$ to $\RR$, such that for every open $\Omega \subset B_{1}$ and $u\in C^{2}(\Omega)$, $I(u,x)$ is continuous.
\end{defn}

This follows \cite[Definition 21]{CS2}; note the weak notion of continuity we assume for any such operator.

Such an operator is uniformly elliptic if it satisfies \eqref{elliptic} for all $u,v$ in $C^{1,1}(x)\cap L^{1}(\RR^{n},\omega_{\sigma})$. Viscosity solutions to \eqref{BVP} are defined in \cite[Definition 25]{CS2}.

We refer to a nonlocal operator as \emph{translation-invariant} if it commutes with translations: letting $\tau_{h}u(x)=u(x-h)$, $I$ is translation-invariant if $\tau_{-h}I(\tau_{h}u,x)=I(u,x)$ for every $x\in B_{1}$ and $u\in C^{1,1}(x)\cap L^{1}(\RR^{n},\omega_{\sigma})$.

For an operator which isn't translation-invariant, we can define the \emph{fixed-coefficient operators} $I_{x_{0}}(u,x)=\tau_{x-x_{0}}I(\tau_{x_{0}-x}u,x)$. These nonlocal operators are translation-invariant and ``freeze'' the $x$ dependence of  $I$ at $x_{0}$.

\subsection{Boundary Regularity}

We will require the following boundary regularity result:
\begin{thm}\label{boundaryreg}
Let $\sigma>\sigma_{0}>0$ and $u$ satisfy the following:
\begin{itemize}
 \item $M^{+}u>-1$ on $B_{1}$
 \item $M^{-}u<1$ on $B_{1}$
 \item $\|u\|_{L^{\infty}(B_{1})}< 1$
 \item $\|u\|_{L^{1}(\RR^{n},\omega)}<1$
 \item $|u(x)-u(y)|<|x-y|^{\beta}$ for $x\in \partial B_{1}$ and $y\in B_{2}\backslash B_{1}$.
\end{itemize}
Then there is an exponent $s>0$ and constant $C$ depending only on $\sigma_{0}, \lambda, \Lambda, n$, and $\beta$ such that 
\begin{equation}
 |u(x)-u(y)|\leq C|x-y|^{s} 
\end{equation}
for all $x\in \bar{B}_{1}$, $y\in B_{2}$.
\end{thm}
\begin{proof}[Sketch of proof]
Let $\eta$ be a smooth cutoff with $\eta \equiv 1$ in $B_{3/2}$ and $\eta \equiv 0$ outside $B_{2}$. Then $|\eta u(x)-\eta u(y)|\leq C|x-y|^{\beta}$ for all $x\in \partial B_{1}$, $y\in B_{2}\backslash B_{1}$. Using \eqref{elliptic},
\[
M^{+}\eta u (x)\geq M^{+}u(x)-M^{+}(1-\eta)u(x)\geq -1-C\|u\|_{L^{1}(\RR^{n},\omega)}
\]
for $x$ in $B_{1}$. A similar computation gives $M^{-}u(x)<C$, so we can apply \cite[Theorem 32]{CS2} with $\rho(z)=C|z|^{\beta}$. That the $\tilde{\rho}$ produced in that theorem is actually \Holder follows by going though the proof and computing the modulus explicitly in each step.
\end{proof}

\subsection{Approximation and Scaling}

In the perturbative argument it will be necessary to have notions of ``closeness'' between nonlocal operators, as well as scale-invariant versions of this notion. We summarize the notation used in \cite{CS2}. A norm on nonlocal operators of order $\sigma$ is given by
\begin{align*} 
 \|I\|=\sup\{\frac{I(u,x)}{1+M}\vert & x\in B_{1}, u\in C^{2}(x), \|u\|_{L^{1}(\RR^{n}.\omega)}\leq M, \\
 & |u(x)-u(y)-(x-y)\nabla u(x)|\leq M|y|^{2} \text{for } y\in B_{1}(x)\}.\\
\end{align*}

The rescaling of the nonlocal operator $I$ is given by $I_{\mu, \gamma}(u,x)=\gamma^{\sigma} \mu I(\tilde{u}/\mu,\gamma x)$ where $\tilde{u}(x)=u(x/\gamma)$. Note that as the extremal operators $M^{+},M^{-}$ are invariant under these scalings, if $I$ is uniformly elliptic with respect to $\cL_{0}$ then so is $I_{\mu, \gamma}$. The scale-invariant norm of $I$ is given by
\begin{equation} \label{normscaled}
 \|I\|_{\sigma}=\sup_{\gamma\leq 1}\|I_{1,\gamma}\|.
\end{equation}

\section{The Constant-Coefficient Estimate}

The aim of this section is to prove Theorem \ref{part1intro}. In fact, we will prove a stronger result with milder but more technical hypotheses. These hypotheses have desirable scaling properties that will be useful here and in the following section. To facilitate the exposition, we introduce the following seminorm for $\beta \leq 1$:
\begin{equation}\label{Aseminorm}
 [u]_{\cA^{\beta}}=\sup_{|h|<1/8}|h|^{-\beta}\int_{\RR^{n}\backslash B_{1+2|h|}}\frac{|u(y+h)-u(y)|dy}{|y|^{n+\sigma_{0}}}.
\end{equation}
This quantity measures the ``$\beta$ derivative'' of $u$ in an $L^{1}$ sense. It is always controlled by the $C^{0,\beta}(B_{1}^{C})$ seminorm. More generally, it is controlled by this weighted \Holder norm, provided $p<\sigma_{0}$ :
\begin{equation} \label{WeightedHolder}
\sup_{|h|<1/8, |x|>1+2|h|}|h|^{-\beta}|u(x+h)-u(x)||x|^{-p}.
\end{equation}

We will also use the following notation for the oscillation:
\begin{equation}
 O[u](\Omega)=\sup\{|u(x)-u(y)|\vert x,y\in \Omega\}.
\end{equation}

\begin{thm}\label{part1}
 Let $\sigma>\sigma_{0}>0$, $I$ be translation-invariant and uniformly elliptic with respect to $\cL_{0}$. Let $u$ satisfy the BVP \eqref{BVP} in the viscosity sense, and assume the following:
\begin{itemize}
 \item[(A)] $\|f\|_{C^{0,1}(B_{1})}<1$
 \item[(B)] $\|u\|_{L^{1}(\RR^{n},\omega)}+\|u\|_{L^{\infty}(B_{1})}<1$
 \item[(C)] $[u]_{\cA^{1}}<1$
 \item[(D)] $\|u\|_{C^{0,1}(B_{2}\backslash B_{1})}<1$
 \item[(E)] $|I(0)|<1$.
\end{itemize}
Then there is an exponent $\alpha$ and constant $C$ depending only on $\sigma_{0}, \lambda, \Lambda,$ and $n$ such that $u\in C^{1,\alpha}(B_{1/2})$ and
\begin{equation}\label{part1c}
 \|u\|_{C^{1,\alpha}(B_{1/2})}\leq C.
\end{equation}
\end{thm}

The proof proceeds by iterating two lemmas. The first uses a weighted \Holder estimate up to the boundary to prove higher regularity in the interior. The second rescales the conclusion of the first in order to obtain a weighted estimate all the way up to the boundary.

\begin{lem}\label{part1lem1}
 Let $I$ be as in Theorem \ref{part1}, $s\leq \beta \leq 1$, and $u$ satisfy the BVP \eqref{BVP}. Assume the following:
\begin{itemize}
 \item[(i)] $\|f\|_{C^{0,\beta}(B_{1})}<1$
 \item[(ii)] $\|u\|_{L^{1}(\RR^{n},\omega)}<\infty$
 \item[(iii)] $|u(x+h)-u(x)|<|h|^{\beta}(1-|x|)^{s-\beta}$ for $|x|<1-2|h|$.
 \item[(iv)] $O[u](B_{2})<1$
 \item[(v)] $[u]_{\cA^{\beta}}<1$
\end{itemize}
Then there is an exponent $\alpha$ and constant $C$ depending only on $s, \sigma_{0}, \lambda, \Lambda,$ and $n$ such that if $\beta+\alpha<1$, $u\in C^{0,\beta+\alpha}(B_{1/2})$ and
\begin{equation} \label{part1lem1c1}
 [u]_{C^{0,\beta+\alpha}(B_{1/2})}\leq C.
\end{equation}
If $\beta+\alpha>1$. we conclude instead that
\begin{equation} \label{part1lem1c2}
 [\nabla u]_{C^{0,\beta+\alpha-1}(B_{1/2})}\leq C.
\end{equation}
\end{lem}

\begin{proof}
For $|x|<3/4$ and $|h|<1/8$ consider the difference quotient
\[
 w_{h}(x)=\frac{u(x+h)-u(x)}{|h|^{\beta}}.
\]
By \cite[Theorem 5.9]{CS1} and (i), we have
\[
 M^{+}w_{h}(x)\geq \frac{f(x+h)-f(x)}{|h|^{\beta}}\geq -[f]_{C^{0,\beta}}>-1.
\]
(Notice that while \cite[Theorem 5.9]{CS1} requires $u$ globally bounded, its proof goes through under our assumption (ii) instead). Likewise, $M^{-}w_{h}(x)<1$.
If we could show  
\begin{equation} \label{goal1}
\|w_{h}\|_{L^{1}(\RR^{n},\omega)}<C, 
\end{equation}
an application of \cite[Theorem 26]{CS2} would give $\|w_{h}\|_{C^{0,\alpha}(B_{1/2})}<C$. A standard lemma then gives the conclusions \eqref{part1lem1c1} and \eqref{part1lem1c2} (see \cite[Lemma 5.6]{CC}).

We proceed to estimate \eqref{goal1} by breaking the domain of integration into three parts:
\begin{align*}
 \|w_{h}\|_{L^{1}(\RR^{n},\omega)} &= \int_{\RR^{n}}\frac{|u(y+h)-u(y)|dy}{|h|^{\beta}(1+|y|)^{n+\sigma_{0}}} \\
 &=\int_{B_{1-2|h|}}+\int_{B_{1+2|h|}\backslash B_{1-2|h|}}+\int_{B_{1+2|h|}^{C}} \\
 &=I_{1}+I_{2}+I_{3}.\\
\end{align*}
The first part is estimated by the weighted estimate (iii):
\[
 I_{1}\leq \int_{B_{1-2|h|}}\frac{|u(y+h)-u(y)|dy}{|h|^{\beta}}\leq \int_{B_{1-2|h|}}(1-|y|)^{s-1}dy\leq C/s.
\]
The second integral is over a region of volume $C|h|$:
\[
 I_{2}\leq C|h|O[u](B_{2})|h|^{-\beta}\leq C
\]
and $I_{3}<C$ immediately from assumption (v). This gives \eqref{goal1} and concludes the proof.
\end{proof}

\begin{lem}\label{part1lem2}
 Take $s\leq \beta < 1$ and make the same assumptions as in Lemma \ref{part1lem1} as well as:
\begin{itemize}
 \item[(vi)] $[u]_{C^{0,\beta}(B_{2}\backslash B_{1})}<1$
 \item[(vii)] $[u]_{C^{0,s}(B_{2})}<1$
\end{itemize}
Then we further have the estimate
\begin{equation} \label{part1lem2c1}
 |u(x+h)-u(x)|<C|h|^{\BA}(1-|x|)^{s-\BA} \text{ for } |x|<1-2|h|.
\end{equation}
\end{lem}

\begin{proof}
 Fix $R$ large and consider the rescaled function $\tilde{u}(x)=u(x/R)$. This satisfies a rescaled equation 
\[
I_{1,1/R}(\tilde{u},x)=R^{-\sigma}f(x/R)\equiv \tilde{f}.
\]
Take $|x_{0}|=R-2$; we wish to apply Lemma \ref{part1lem1} to $R^{s}\tilde{u}$ on $B_{1}(x_{0})$. To that end, we check assumptions (i)-(v) on the rescaled equation. 

For (i), use 
\[
\|\tilde{f}\|_{C^{0,\beta}}\leq R^{-\sigma}\|f\|_{C^{0,\beta}}\leq R^{-\sigma}. 
\]
(ii) is immediate since the actual value of the norm is irrelevant. For (iii), we use the fact that $|h|<1/2\leq(R-|x|)/2$ (since ($R-|x_{0}|=2$):
\[
 |\tilde{u}(x+h)-\tilde{u}(x)|\leq |h/R|^{\beta}(1-|x/R|)^{s-\beta}\leq R^{-s}|h|^{\beta}(R-|x|)^{\beta}\leq CR^{-s}|h|^{\beta},
\]
using the fact we assumed (iii) on the non-scaled equation. (iv) can be estimated by means of assumption (vii):
\[
 O[\tilde{u}_{x_{0}}](B_{2})\leq O[u](B_{2/R}(x_{0}/R))\leq CR^{-s}
\]
where $\tilde{u}_{x_{0}}(x)=\tilde{u}(x-x_{0})$.

Assumption (v) takes more work to check. We split each integral into four regions:
\begin{align*}
 &\int_{|y-x_{0}|>1+2|h|}\frac{|\tilde{u}(y+h)-\tilde{u}(y)|}{|y-x_{0}|^{n+\sigma_{0}}}dy \\
&\leq \int_{B_{R-2|h|}(0)\backslash B_{1}(x_{0})}
+\int_{B_{R+2|h|}(0)\backslash B_{R-2|h|}(0)}
+\int_{B_{2R}(0)\backslash B_{R+2|h|}(0)}
+\int_{B_{2R}^{C}(0)} \\
&= I_{1}+I_{2}+I_{3}+I_{4}.\\
\end{align*}
$I_{4}$ is estimated by rescaling assumption (v):
\begin{align*}
 I_{4} &\leq \int_{|y|>2R}\frac{|\tilde{u}(y+h)-\tilde{u}(y)|}{|y|^{n+\sigma_{0}}}\left(\frac{|y|}{|y-x_{0}|}\right)^{n+\sigma_{0}}dy \\
&\leq C \int_{|y|>2}\frac{|u(y+h/R)-u(y)|}{|yR|^{n+\sigma_{0}}}R^{n}dy \\
&\leq C[u]_{\cA^{\beta}}|h|^{\beta}R^{-\sigma_{0}-\beta} \leq C|h|^{\beta}R^{-s}. \\
\end{align*}
For $I_{3}$ we use (vi):
\begin{align*}
 I_{3} &= \int_{R+2|h|<|y|<2R}\frac{|\tilde{u}(y+h)-\tilde{u}(y)|}{|x_{0}-y|^{n+\sigma_{0}}}dy\\
&\leq C \int_{R+2|h|<|y|<2R}|h/R|^{\beta}|x_{0}-y|^{-n-\sigma_{0}}dy \\
&\leq C \int_{1<|y-x_{0}|}|h/R|^{\beta}|x_{0}-y|^{-n-\sigma_{0}}dy 
\leq C|h|^{\beta}R^{-s}.\\
\end{align*}
For the next integral we rescale the oscillation:
\begin{align*}
 I_{2}&\leq \int_{R-2|h|<|y|<R+2|h|}\frac{O[\tilde{u}_{x}](B_{h})}{|x_{0}-y|^{n+\sigma_{0}}}dy \\
&\leq C|h/R|^{s} \int_{R-2|h|<|y|<R+2|h|}\frac{1}{|x_{0}-y|^{n+\sigma_{0}}}dy, \\
\end{align*}
which can be estimated explicitly. The following is easy to see using polar coordinates about $0$.
\[
 \int_{R-2|h|<|y|<R+2|h|}\frac{1}{|x_{0}-y|^{n+\sigma_{0}}}dy
\leq C|h|\int_{|y|=R}\frac{1}{|x_{0}-y|^{n+\sigma_{0}}}dy
\]
Without loss of generality take $x_{0}=(R-2,0,0,\ldots)$. Now we estimate the integral over the hemisphere away from $x_{0}$:
\[
 C|h|\int_{|y|=R, y_{1}<0}\frac{1}{|x_{0}-y|^{n+\sigma_{0}}}dy\leq C|h|R^{n-1}R^{-n-\sigma_{0}}.
\]
 For the other hemisphere, change variables to the stereographic projection from the point $(-R,0,0,\ldots)$, i.e. 
\begin{equation}\label{stereo}
y=(y_{1},y_{2},\ldots,y_{n})\in \partial B_{R} \rightarrow z=\left(\frac{y_{2}}{R+y_{1}},\ldots,\frac{y_{n}}{R+y_{1}}\right)\in \RR^{n-1}.
\end{equation}
In these coordinates, $|x_{0}-y|>C\sqrt{|Rz|^{2}+4}$ provided $y_{1}>0$. The volume element is $dy=R^{n-1}/(1+|z|^{2})^{n-1}dz$. putting all this together, the remaining part of the integral is
\[
\leq C|h|\int_{|z|\leq 1}\frac{R^{n-1}}{(1+|z|^{2})^{n-1}(1+|Rz|)^{n+\sigma_{0}}}dz,
\]
and (changing variables) that integral is bounded by
\[
 \int_{\RR^{n-1}}\frac{1}{(1+|z/R|^{2})^{n-1}(1+|z|)^{n+\sigma_{0}}}dz\leq C.
\]

The final part remaining, $I_{1}$, works similarly, but using the weighted estimate (iii). We break the integral into two parts:
\begin{align*}
 I_{1} &\leq \int_{|y|<R-2|h|, |y-x_{0}|>1}\frac{|h/R|^{\beta}(1-|y/R|)^{s-\beta}}{|x_{0}-y|^{n+\sigma_{0}}}dy \\
&\leq CR^{-s}|h|^{\beta}\int_{|y|\leq R}\frac{(R-|y|)^{s-\beta}}{(1+|y-x_{0}|)^{n+\sigma_{0}}}dy \\
&\leq CR^{-s}|h|^{\beta}\left(\int_{R-1<|y|<R} +\int_{|y|<R-1} \right). \\
\end{align*}
Over the region $|y|<R-1$, the weight is bounded by $1$, so we have
\[
 \int_{|y|<R-1}\frac{(R-|y|)^{s-\beta}}{(1+|y-x_{0}|)^{n+\sigma_{0}}}dy
\leq \int_{\RR^{n}}\frac{1}{(1+|y-x_{0}|)^{n+\sigma_{0}}}dy<C
\]
Over the other region, use the triangle inequality to give $|y-x_{0}|\geq |Ry/|y|-x_{0}|-|Ry/|y|-y|\geq |Ry/|y|-x_{0}|-1$. Plugging this in and switching to polar coordinates about $0$,
\[
 \int_{R-1<|y|<R}\frac{(R-|y|)^{s-\beta}}{(1+|y-x_{0}|)^{n+\sigma_{0}}}dy
\leq \int_{|y|=R}\frac{1}{|y-x_{0}|^{n+\sigma_{0}}}dy\int_{R-1}^{R}(R-r)^{s-\beta}dr.
\]
The first factor we've already estimated above using \eqref{stereo}, while the second factor is finite. This shows that $I_{1}\leq CR^{-s}|h|^{\beta}$.

We can now invoke Lemma \ref{part1lem1} on the function $R^{s}\tilde{u}_{x_{0}}$. Set $x_{0}=Rz$ with $1-2/R=|z|$. Conclusion \eqref{part1lem1c1} implies that if $|h|<1/2R=(1-|z|)/4$, 
\begin{align*}
|u(z+h)-u(z)|&= |\tilde{u}(Rz+Rh)-\tilde{u}(Rz)|\\
&\leq CR^{-s}|Rh|^{\BA} \leq C |h|^{\BA}(1-|z|)^{s-\BA}.
\end{align*}
At the cost of a larger constant, the same will hold for $|h|<(1-|z|)/2$. This gives \eqref{part1lem2c1} and completes the proof.
\end{proof}

\begin{proof}[Proof of Theorem \ref{part1}]
First, use (A), (B), (D), and (E) with Theorem \ref{boundaryreg} to obtain $u \in C^{0,s}(B_{2})$ for some $s$. We can now apply Lemmas \ref{part1lem1} and \ref{part1lem2} with $\beta=s$. Indeed, (i) follows from (A), (ii) from (B), (iii), (iv), and (vii) from the boundary regularity just mentioned, (v) from (C), and (vi) from (D).

Now apply both lemmas with $\beta=s + \alpha$. All the assumptions are satisfied for the same reason except (iii), which follows from \eqref{part1lem2c1} from the previous step. Iterate this until the first time $\beta+\alpha>1$, and then use Lemma \ref{part1lem1} one final time. Note that this requires only a finite, universal, number of steps (at most $1/\alpha+1$), and so the growth of the constants in each step is not a problem.
\end{proof}

\begin{rem}\label{part1rem}
 Note that if in Theorem \ref{part1} we assume
\begin{itemize}
 \item[(A')]$\|f\|_{C^{0,\beta}(B_{1})}<1$
 \item[(C')]$[u]_{\cA^{\beta}}<1$
 \item[(D')]$\|u\|_{C^{0,\beta}(B_{2}\backslash B_{1})}<1$
\end{itemize}
in place of (A), (C), and (D), the proof still goes through up to applying Lemma \ref{part1lem1} with exponent $\beta$, giving the conclusion
\begin{equation}\label{part1remc}
  \|u\|_{C^{\beta+\alpha}(B_{1/2})}\leq C
\end{equation}
in place of \eqref{part1c} (provided $\beta+\alpha \neq 1$.
\end{rem}

In later sections it will be helpful to have the following quantitative version of Theorem \ref{part1}.

\begin{cor}\label{part1cor}
Under the assumptions of Theorem \ref{part1}, we have the further estimate
\begin{equation}\label{part1corc1}
 |\nabla u(x+h) - \nabla u(x)|\leq C |h|^{\alpha'}(1-|x|)^{s-\alpha'-1}
\end{equation}
for $|h|<(1-|x|)/2$ and $0<\alpha'\leq \alpha)$.
\end{cor}

\begin{proof}
 \eqref{part1corc1} is an immediate consequence of applying the proof of Lemma \ref{part1lem2} with $\beta=1$.
\end{proof}

\section{Perturbative Argument}

\renewcommand{\BA}{\beta+\alpha}

In this section we prove Theorem \ref{part2intro}, in the following more general form:

\begin{thm} \label{part2}
 Let $I$ be uniformly elliptic with respect to $\cL_{0}$, $\sigma>\sigma_{0}>1$, and $u$ satisfy $-C_{0}\leq I(u,x)\leq C_{0}$ in the viscosity sense on $B_{1}$. Assume $u\in L^{1}(\RR^{n},\omega)$. Then there exist $\alpha_{1}, \eta>0$ such that if $\|I_{x_{0}}-I\|_{\sigma}< \eta$ for each $x_{0}\in B_{1/2}$, then we have
\begin{equation}
 \|u\|_{C^{1,\alpha_{1}}(B_{1/2})}\leq C \left(\|u\|_{L^{1}(\RR^{n},\omega)}+\|u\|_{L^{\infty}(B_{1})}+C_{0}+\|I(0,x)\|_{L^{\infty}(B_{1})}\right).
\end{equation}
\end{thm}

By $I_{x_{0}}$ we mean the operator with fixed coefficients. Note that there is no loss of generality in assuming $I(0,x)=0$, for otherwise consider the operator $I(u,x)-I(0,x)$ instead, which satisfies all of the same assumptions. In what follows, let $\alpha$ be one quarter the minimum of the exponent in the conclusion of Theorem \ref{part1} and $\sigma_{0}-1$. 

The structure of the argument is as follows: rather than prove the interior $C^{1,\alpha}$ estimate immediately, we instead prove a sequence of $C^{0,\beta}$ interior estimates, each time increasing $\beta$ a fixed amount. Then a final step takes us from a $C^{0,\beta}$ estimate with $\beta$ very close to $1$ to the desired $C^{1,\alpha}$ estimate. We begin by proving the following lemma by induction on the parameter $\beta$.

\begin{lem}\label{part2lem1}
 Fix parameter $\beta\in (0,1)$. Under the same assumptions as Theorem \ref{part2}, we have 
\begin{equation} \label{part2lem1c}
 \|u\|_{C^{0,\beta}(B_{1/2})}\leq C \left(\|u\|_{L^{1}(\RR^{n},\omega)}+\|u\|_{L^{\infty}(B_{1})}+C_{0}\right)
\end{equation}
for a constant $C=C(\beta)$.
\end{lem}

Note that the constant in \eqref{part2lem1c} is allowed to depend on $\beta$; in particular, it may blow up as $\beta \rightarrow 1$. In the proof of Theorem \ref{part2}, we'll only use Lemma \ref{part2lem1} for some fixed parameter $\beta<1$, and won't care about this limiting behavior. Of course, once Theorem \ref{part2} is proven,  \eqref{part2lem1c} is valid with a constant independent of $\beta$.

We break up the proof of Lemma \ref{part2lem1} into parts. The first step is the following lemma, where notably we assume that $u$ is in $C^{\beta}(\RR^{n})$:

\begin{lem} \label{part2lem2}
 Assume Lemma \ref{part2lem1} holds with parameter $\beta<1$. Then there exist $\eta, C$ depending only on $\beta, \lambda, \Lambda, \sigma_{0}$, and $n$ such that if $-\eta \leq I(u,x)\leq \eta$ on $B_{1}$, with $I$ as in Theorem \ref{part2}, $\beta+\alpha<1$, and $\|u\|_{C^{0,\beta}(\RR^{n})}\leq 1$, then $|u(x)-u(0)|<C|x|^{\BA}$ on $B_{1/2}$.
\end{lem}

\begin{proof}
Let $\gamma$ be a (small) constant to be chosen below. Choose $\eta$ so as to make \cite[Lemma 7]{CS2} apply with $\rho(z)=2|z|^{\beta}+2|z|^{\beta+2\alpha}$, $\epsilon=\gamma^{\beta+3\alpha}$, and $M=1$. We construct a sequence of numbers $a_{k}$ with the following properties:
\begin{itemize}
 \item[(i.k)] $\sup_{B_{\gamma^{k}}}|u(x)-a_{k}|\leq \gamma^{k(\BA)}$
 \item[(ii.k)] $|u(\gamma^{k}x)-a_{k}|\leq \gamma^{k(\BA)}(1+|x|^{\beta+2\alpha})$ for $x\in \RR^{n}$
 \item[(iii.k)] $|u(\gamma^{k}(x+h))-u(\gamma^{k}(x))|\leq |h|^{\beta}\gamma^{k(\BA)}(1+|2x|^{\beta+2\alpha})$ for $|h|<1$ and $x\in \RR^{n}$.
\end{itemize}
The conclusion then follows, as it's easy to check that $|a_{k+1}-a_{k}|<2\gamma^{k(\BA)}$, $\lim a_{k}=u(0)$, and $|u(x)-u(0)|\leq 2\gamma^{k(\BA)}$ for $|x|<\gamma^{k}$. Heuristically, (i) controls the oscillation of $u$ on smaller balls at the correct rate to recover the desired $|x|^{\BA}$ modulus, while (ii) and (iii) are artifacts of the nonlocal nature of the operator, and are needed to control the growth and smoothness of the tails of $u$ at each step in the iteration.

Set $a_{0}=0$; then (i.0)-(iii.0) are immediate. Now given $a_{k}$ satisfying (i.k), (ii.k), and (iii.k), we construct $a_{k+1}$ satisfying (i.k+1), (ii.k+1) and (iii.k+1). Define
\[
 w_{k}(x)=\frac{u(\gamma^{k}x)-a_{k}}{\gamma^{k(\BA)}}.
\]
This satisfies $|I_{k}(w_{k},x)|\leq \eta \gamma^{k(\sigma-\beta-\alpha)}$ for $|x|<\gamma^{-k}$, where $I_{k} \equiv I_{\gamma^{-k(\BA)},\gamma^{k}}(w_{k},x)$.  Notice that the right-hand side remains bounded by $\eta$ (i.e. it gets smaller as $k$ increases). Now let $I_{k}^{0}$ be the operator with coefficients fixed at $0$, and solve
\[
 \begin{cases}
  I^{0}_{k}(v_{k},x)=0 & x \in B_{1} \\
  v_{k}=w_{k} & x \notin B_{1}.
 \end{cases}
\]
This is the Dirichlet problem for a uniformly elliptic translation-invariant operator $I^{0}$, and so it admits a unique solution. As $\|I_{k}-I_{k}^{0}\|\leq \eta$ by assumption, we use (ii.k) and (iii.k) to see that \cite[Lemma 7]{CS2} applies. This gives 
\begin{equation} \label{A1}
|w_{k}-v_{k}|\leq \gamma^{\beta+3\alpha}.
\end{equation}

At the same time, notice that $\|w_{k}\|_{L^{1}(\RR^{n},\omega)}\leq C$ by (ii.k) and $[w_{k}]_{\cA^{\beta}}+\|w_{k}\|_{C^{0,\beta}(B_{2})}\leq C$ by (iii.k), so using Remark \ref{part1rem} we have
\begin{equation} \label{A2}
\|v_{k}\|_{C^{0,1 \wedge (\beta+3\alpha)}(B_{1/2})}\leq C_{1}.
\end{equation}

Putting these together gives the following two inequalities. For $x\in B_{1/2}$.
\begin{align}
  |w_{k}(x)-v_{k}(0)| &\leq |w_{k}(x)-v_{k}(x)| + |v_{k}(x)-v_{k}(0)| \notag\\
  &\leq \gamma^{\beta+3\alpha} + C_{1}|x|^{1 \wedge (\beta+3\alpha)},  \label{B1}
\end{align}
using \eqref{A1} and \eqref{A2}. Everywhere else (in fact, for any $x\in \RR^{n}$), we can estimate this differently: 
\begin{align}
 |w_{k}(x)-v_{k}(0)| &\leq |w_{k}(x)-w_{k}(0)| + |w_{k}(0)-v_{k}(0)| \leq |w_{k}(x)|+|w_{k}(0)|+|w_{k}(0)-v_{k}(0)| \notag\\
 &\leq  (1+|x|^{\beta+2\alpha}) + 1 + \gamma^{\beta+3\alpha}, \label{B2}
\end{align}
this time using (i.k), (ii.k), and \eqref{A1}.

Set $a_{k+1}=a_{k}+\gamma^{k(\BA)}v_{k}(0)$. Then $w_{k+1}$ verifies the formula $w_{k+1}(x)=\frac{w_{k}(\gamma x)-v_{k}(0)}{\gamma^{\BA}}$. We can now use this formula to get estimates on $w_{k+1}$ in several different regions. For $x\in B_{1}$,
\begin{align}
   |w_{k+1}(x)| = &\gamma^{-\beta-\alpha}|w_{k}(\gamma x)-v_{k}(0)| \notag \\
&\leq \gamma^{2\alpha} + C_{1}|\gamma x|^{1\wedge(\beta+3\alpha)}\gamma^{-\beta-\alpha} \notag\\
&\leq (1+C_{1})\gamma^{2\alpha\wedge(1-\beta-\alpha)} \leq m \label{C1},
\end{align}
for $m<1$ small to be chosen below. We use \eqref{B1}, the fact that $|x|<1$, and in the last step made sure $\gamma=\gamma(m)$ is chosen sufficiently small (since $\beta+\alpha<1$ was assumed). This immediately gives (i.k+1). 

For $x$ in the annulus $B_{1/2\gamma}\backslash B_{1}$, we can still use \eqref{B1}:
\begin{align}
   |w_{k+1}(x)| &\leq \gamma^{2\alpha} + C_{1}|\gamma x|^{1\wedge(\beta+3\alpha)}\gamma^{-\beta-\alpha} \notag\\
&\leq (1+C_{1})\gamma^{\alpha\wedge(1-\beta-\alpha)}(1+|x|^{\beta+2\alpha}) \leq m(1+|x|^{\beta+2\alpha}) \label{C2},
\end{align}
where to get to the second line, we use $|x|>1$ if $\beta+3\alpha>1$ and $|x|<1/\gamma$ otherwise (and then, again, chose $\gamma$ sufficiently small).

When $1/2\gamma \leq |x|$, we are forced to use \eqref{B2}:
 \begin{align}
   |w_{k+1}(x)| &\leq [2+|\gamma x|^{\beta+2\alpha} + \gamma^{\beta+3\alpha}]\gamma^{-\beta-\alpha} \notag\\
&\leq \gamma^{2\alpha} +3\gamma^{\alpha}|2x|^{\beta+2\alpha} \leq m(1+|x|^{\beta+2\alpha}) \label{C3}, 
\end{align} 
using that $|x|>1/2\gamma$ to estimate the first term and choosing $\gamma$ even smaller if need be. Now combining \eqref{C1}, \eqref{C2}, and \eqref{C3} gives (ii.k+1).

To get (iii.k+1), we apply Lemma \ref{part2lem1} with parameter $\beta$ (recall that this was assumed) to $w_{k+1}$, From the computation above $|I_{k+1}(w_{k+1},x)|\leq \eta$ on $B_{1/\gamma^{k+1}}$. Picking $y\in B_{1/2\gamma^{k+1}}$, 
\[
\|w_{k+1}(\cdot - y)\|_{L^{1}(\RR^{n},\omega)}\leq \int \frac{m(1+|x-y|^{\beta+2\alpha})dx}{(1+|x|)^{n+\sigma_{0}}}
 \leq mC(1+|y|^{\beta+2\alpha}).
\]
It follows, then, (choosing $m$ small and possibly making $\eta$ smaller) that $\|w_{k+1}\|_{C^{0,\beta}(B_{1}(y))}\leq (1+|y|^{\beta+2\alpha})$ for $y\in B_{1/2\gamma^{k+1}}$. For $|y|>1/2\gamma^{k+1}$ rescale the original assumption to give 
\begin{align*}
|w_{k+1}(y+h)-w_{k+1}(y)|&=|u(\gamma^{k+1}(y+h))-u(\gamma^{k+1}y)|\gamma^{-(k+1)(\beta+\alpha)} \\
&\leq \gamma^{-(k+1)\alpha}|h|^{\beta}\leq |h|^{\beta} |2y|^{\alpha}\leq |h|^{\beta}(1 + |2y|)^{\beta+2\alpha}.
\end{align*}
 Together these imply (iii.k+1).
\end{proof}

A note on our use of \cite[Lemma 7]{CS2}. The modulus in (iii.k) deteriorates for larger $|x|$. Whether the statement of the Lemma accounts for this situation (an $x-$dependent modulus) is a bit unclear, but the proof clearly goes through.

We are now in a position to prove Lemma \ref{part2lem1} for every parameter $\beta$. This consists of a reduction argument and an application of Lemma \ref{part2lem2} in the inductive step.

\begin{proof}[Proof of Lemma \ref{part2lem1}]
 Recall that $\alpha$ is the smaller of $(\sigma_{0}-1)/4$ and one quarter of the exponent in Theorem \ref{part1}, or equivalently in \cite[Theorem 26]{CS2}. Then for $\beta=4\alpha$, \eqref{part2lem1c} follows from \cite[Theorem 26]{CS2}. Assume the lemma holds for parameter $\beta$; we'll prove it for parameter $\beta+\alpha<1$.

Let $\phi$ be a smooth cutoff with $\phi \equiv 1$ in $B_{1/4}$ and supported on $B_{1/2}$. Using Lemma \ref{part2lem1} with parameter $\beta$ (which we just assumed true), we have that $\phi u \in C^{0,\beta}(\RR^{n})$. But also, for $x \in B_{1/8}$, we claim that 
\[
I(\phi u, x)\leq I(u,x)+M^{+}(\phi-1)u(x)\leq C_{0} + C\|u\|_{L^{1}(\RR^{n},\omega)},
\]
which is assumed bounded. We must justify the first inequality (since $I$ is not translation invariant and $u$ solves the equation only in the viscosity sense). For $r<1/8$ and every $v\in C^{2}(B_{r}(x))$ satisfying $v\leq u$, $v(x)=u(x)$,  set 
\[
\psi(y)=\begin{cases}
  v(y) & B_{r}(x) \\
  \phi u(y) & \RR^{n}\backslash B_{r}(x)
 \end{cases}
\]
which is a test function for $\phi u$. But notice that $\psi^{*}$, defined by
\[
\psi^{*}(y)=\begin{cases}
  v(y) & B_{r}(x) \\
  u(y) & \RR^{n}\backslash B_{r}(x),
 \end{cases}
\]
is a test function for $u$. That means that

\[
 I(\psi,x)\leq I(\psi^{*},x)+M^{+}(\psi-\psi^{*})(x)\leq C_{0} + M^{+}(\phi-1)u(x),
\]
where the second step uses that $\psi^{*}$ is a test function for $u$. Since this inequality is satisfied for all suitable test functions $\psi$ touching $\phi u$ from below at $x$, the inequality
\[
 I(\phi u,x)\leq C_{0} + M^{+}(\phi-1)u(x) \leq C_{0}+C\|u\|_{L^{1}(\RR^{n},\omega)}
\]
is satisfied in the viscosity sense for $|x|<1/8$. Likewise, in the same region we have the other inequality
\[
 I(\phi u,x)\geq -C_{0} + M^{-}(\phi-1)u(x) \geq -C_{0}-C\|u\|_{L^{1}(\RR^{n},\omega)}.
\]

Now rescale and translate: for $|z|<1$ and $\mu>0$ take $q_{z}$ to be defined by
\[
 q_{z}(x)= \mu \phi u(\frac{x-z}{16}).
\]
Then $q_{z}$ satisfies $|I_{\mu,1/16}(q_{z},x)|\leq C \mu (1/16)^{\sigma}\leq \eta$ for $|x|<1$, $\eta$ as in Lemma \ref{part2lem2}, and $\mu$ chosen sufficiently small. By choosing $\mu$ smaller if need be, we can also ensure $\|q_{z}\|_{C^{0,\beta}(\RR^{n})}\leq 1$ and $\|q_{z}\|_{L^{1}(\RR^{n},\omega)}\leq 1$. Applying Lemma \ref{part2lem2}, we obtain that $|q_{z}(x)-q_{z}(0)|\leq C|x|^{\beta+\alpha}$ for $|x|<1/2$. Scaling back gives that $\|\phi u\|_{C^{0,\beta+\alpha}(B_{1/16})}=\|u\|_{C^{0,\beta+\alpha}(B_{1/16})}\leq C'$, as desired. (The linear dependence on the hypotheses follows from scaling, while having the estimate on $B_{1/2}$ instead of $B_{1/16}$ follows from a covering argument.)
\end{proof}

We now complete the argument:

\begin{proof}[Proof of Theorem \ref{part2}]
 The procedure is the same as the proof of Lemmas \ref{part2lem1} and \ref{part2lem2}. Make the same reductions as in the proof of Lemma \ref{part2lem1}: in particular, by using the lemma with parameter $\beta=1-\alpha/2$, reduce to the situation where $\|u\|_{C^{0,1-\alpha/2}(\RR^{n})}\leq 1$ and $|I(u,x)|\leq \eta$ on $B_{1}$. Proceeding as in the proof of Lemma \ref{part2lem2}, rather than a sequence of numbers, we construct a sequence of linear functions $l_{k} = a_{k}+b_{k}\cdot x$ satisfying 
\begin{itemize}
 \item[(i.k)] $\sup_{B_{\gamma^{k}}}|u(x)-l_{k}(x)|\leq \gamma^{k(1+\alpha)}$
 \item[(ii.k)] $|u(\gamma^{k}x)-l_{k}(\gamma^{k}x)|\leq \gamma^{k(1+\alpha)}(1+|x|^{1+2\alpha})$
 \item[(iii.k)] $|u(\gamma^{k}(x+h))-u(\gamma^{k}(x))|\leq |h|^{1-\alpha/2}\gamma^{k(1+\alpha)}(1+|2x|^{1+2\alpha})$ for $|h|<1$
 \item[(iv.k)] $|a_{k+1}-a_{k}|+\gamma^{k}|b_{k+1}-b_{k}|\leq 2C_{1}\gamma^{k(1+\alpha)}$
\end{itemize}
 Set $w_{k+1}=\frac{u(\gamma^{k}x)-l_{k}(\gamma x)}{\gamma^{k(1+\alpha)}}$. Then construct $v_{k}$ as before and set $\bar{l}_{k}=v_{k}(0)+\nabla v_{k}(0)\cdot x$ and $l_{k+1}(x)=l_{k}(x)+\gamma^{k(1+\alpha)}\bar(l)(x/\gamma^{k})$. (iii.k) ensures $w_{k}$ is in $\cA^{1-\alpha/2}$, and so we get the estimate $v_{k}\in C^{1,3\alpha}(B_{1/2})$ from Remark \ref{part1rem}. (i.k+1), (ii.k+1), and (iii.k+1) follow as before, while (iv.k+1) follows from
\[
 |a_{k+1}-a_{k}|=\gamma^{k(1+\alpha)}|\bar{l}(0)|=\gamma^{k(1+\alpha)}|v_{k}(0)|\leq C_{1}\gamma^{k(1+\alpha)}
\]
and
\[
 |b_{k+1}-b_{k}|=\gamma^{k\alpha}|\nabla \bar{l}(0)|=\gamma^{k\alpha}|\nabla v_{k}(0)|\leq C_{1}\gamma^{k\alpha}
\]
where $C_{1}$ is the constant from Theorem \ref{part1}. We conclude that $u$ is differentiable at $0$ and $|u(x)-u(0)-x\cdot \nabla u(0)|\leq C|x|^{1+\alpha}$. Together with the reduction, this implies Theorem \ref{part2} with $\alpha_{1}=\alpha$.
\end{proof}

\begin{rem}\label{part2rem}
 The same argument works if instead of $|I(u,x)|\leq C_{0}$ we have $I^{(1)}(u,x)\geq C_{0}$ and $I^{(2)}(u,x) \leq C_{0}$ on $B_{1}$, with $\|I^{(1)}-I^{(0)}\|_{\sigma}+\|I^{(2)}-I^{(0)}\|_{\sigma}\leq \eta$ and $I^{0}$ translation-invariant.
\end{rem}

\begin{rem}\label{part2rem2}
 If it's known that in addition to the assumptions above the boundary data is \Holder ($u\in C^{0,\alpha}(B_{1}^{C})$), 
the scaling argument in Lemma \ref{part1lem2} 
can be combined with Theorem \ref{part2} 
to give $\nabla u \in L^{1}(B_{1})$.
\end{rem}

\section{Existence and Smooth Approximations}

In this section we assemble some results for the Dirichlet problem \eqref{BVP}. In the local setting, existence and uniqueness are simple consequences of a comparison principle, developed for viscosity solutions by Jensen \cite{J}. For nonlocal equations, a comparison principle was proved in \cite{CS1} for translation-invariant equations, and in \cite{BI} for some classes of nonlocal operators sufficiently regular in $x$.

Rather than proving a more general comparison principle here, we approach the problem from a different and decidedly nonlocal perspective. A remarkable fact about uniformly elliptic nonlocal operators is that they can be approximated by operators which admit smooth solutions. This was used in \cite{CS3}, and is in stark contrast to the second-order theory. Below, we show how to apply this to \eqref{BVP} for general $I$.

Fix $\sigma>\sigma_{0}>1$. We will assume for the rest of this section that $I$ satisfies the following stronger continuity property:
\begin{equation}\label{strongcont}
 \|I_{x}-I_{x_{0}}\| \leq \rho(|x-x_{0}|) \quad \forall x,x_{0}\in B_{1},
\end{equation}
where $\rho$ is a modulus of continuity. Notice this is more restrictive than the continuity mandated in Definition \ref{nonlocalopdef}. We also assume $I(0,x)=0$; the theorem below will also be true for $I(0,x)$ bounded and continuous.

Our goal is the following:
\begin{thm}\label{ExUnqThm}
 Consider the Dirichlet problem \eqref{BVP} with $f\in C(B_{1})\cap L^{\infty}(B_{1})$ and $g\in L^{1}(\RR^{n},\omega)$ with $|g(x)-g(y)|\leq \rho'(|x-y|)$ for $x\in \partial B_{1}$, $y\in B_{2}\backslash B_{1}$, where $\rho'$ is a modulus of continuity. Then there exists a viscosity solution $u$ which is bounded on $B_{1}$.
\end{thm}

Our first task is to construct a family of equations $I^{\epsilon}$ approximating $I$ and admitting classical solutions. The idea is to replace the operator with a fractional Laplacian on a small ball, as in \cite{CS3}.

First, we observe that $I$ is truly a function of the symmetric differences $\delta_{2}u$. Indeed, any two measurable functions whose symmetric differences coincide differ by an affine function $l$, and so, if they are in $C^{1,1}(x)\cap L^{1}(\RR^{n},\omega)$, using ellipticity gives
\[
 0=M^{-}(l,x)\leq I(u+l,x)-I(u,x)\leq M^{+}(l,x)=0.
\]
This is one way of using the nondivergence structure of these operators. 

Consider operators $J(w,x)$ acting on even functions $w:\RR^{n}\rightarrow \RR$ with $w(0)=0$ and $\psi \in C^{1,1}(0)\cap L^{1}(\RR^{n},\omega)$. Every $J$ leads to a nonlocal operator $I$ via
\begin{equation}\label{IfromJ}
 I(u,x)=J(\delta_{2}u(x,\cdot),x).
\end{equation}
Conversely, every even function with $w(0)=0$ can be written as the symmetric difference of a function (fixed at $x$); indeed, the formula $u(x+y)=w(y)/2$ gives the unique such $u$ symmetric about the point $x$ and with $u(x)=0$. Using this and the observation above, we can define a $J$ for every nonlocal operator $I$ by
\begin{equation}\label{JfromI}
 J(w,x)=I(w(x-\cdot)/2,x).
\end{equation}
These procedures give a bijection between the $I$ and the $J$. All of the constructions for nonlocal operators carry over to $J$ in the obvious ways (e.g uniform ellipticity, $J_{\mu,\lambda}$).  We use the notation $N^{+}$ for the $J$ corresponding to the extremal operator $M^{+}$.

Let $\phi$ be a smooth radial cutoff function $\equiv 1$ on $B_{1/2}$ and supported on $B_{1}$, and set $\phi_{\epsilon}(x)=\phi(x/\epsilon)$. Let $\eta_{\epsilon}(x)$ be a standard mollifier; in other words, $\eta_{\epsilon}$ integrates to $1$, is smooth, positive, radial, and supported on a ball of radius $\epsilon$. Given a nonlocal operator in form $J$, recall that $J_{z}$ is the operator frozen at $z$, i.e. $J_{z}(w,x)=\tau_{x-z}J(\tau_{z-x}w,x)$ where $\tau_{h}v(x)=v(x-h)$. Define 
\begin{equation*}
 \tilde{J}^{\epsilon}_{z}(u,x)=\begin{cases}
  J_{z}(u,x) & |z|\leq 1-\epsilon \\
  J_{(1-\epsilon)z/|z|}(u,x) & |z|>1-\epsilon.\\ 
 \end{cases}
\end{equation*}
Observe that $\tilde{J}^{\epsilon}_{z}$ is uniformly elliptic for each $z$ if $J$ is. Then construct
\begin{equation}\label{approxJ}
 J^{\epsilon}(w,x)=\tilde{J}^{\epsilon}_{z}(w(1-\phi_{\epsilon}),x)\ast_{z} \eta_{\epsilon}(x)+(2-\sigma)\int\frac{\lambda \phi_{\epsilon}(y)w(y)}{|y|^{n+\sigma}}dy
\end{equation}
for $|x|< 1$. The notation is clumsy; this is what the first term actually is: 
\[
 \int \eta_{\epsilon}(x-z)\tilde{J}^{\epsilon}_{z}(w(1-\phi_{\epsilon}),x)dz.
\]

There are two regularizations being made in this construction: cutting off the kernels and replacing by the fractional Laplacian near the origin, via $\eta$, and a smoothing in $x$, via $\eta_{\epsilon}$.

Using a smooth cutoff $\phi$ allows the operator to preserve regularity in $y$, although we won't require this below.

\begin{lem}
If $J$ is uniformly elliptic with respect to $\cL_{0}$, so is $J^{\epsilon}$. If $I$ is, in addition, continuous in the sense \eqref{strongcont},  $\|I^{\epsilon}-I\|\rightarrow 0$ as $\epsilon \rightarrow 0$. 
\end{lem}

\begin{proof}
First the uniform ellipticity:
\begin{align*}
 J^{\epsilon}(w,x)-J^{\epsilon}(w',x) &=[\tilde{J}^{\epsilon}_{z}(w(1-\phi_{\epsilon}),x)-\tilde{J}^{\epsilon}_{z}(w'(1-\phi_{\epsilon}),x)]\ast_{z} \eta_{\epsilon}(x) \\& +(2-\sigma)\int\frac{\lambda \phi_{\epsilon}(y)[w(y)-w'(y)]}{|y|^{n+\sigma}}dy \\
&\leq (2-\sigma)\int\frac{\Lambda [(w-w')(1-\phi_{\epsilon})]^{+} -\lambda [(w-w')(1-\phi_{\epsilon})]^{-}+ \lambda \phi_{\epsilon}[w-w']}{|y|^{n+\sigma}}dy \\
&\leq N^{+}(w-w',x),\\
\end{align*}
using \eqref{JfromI}, ellipticity of $\tilde{J}^{\epsilon}_{z}$, and the explicit form of $M^{+}$. Note that as $M^{+}$ is translation invariant, the convolution becomes irrelevant.  The same can be done for $M^{-}$.

For the convergence, take $u$ an admissible function for $I$ and $I^{\epsilon}$; in other words, in $C^{1,1}(x)\cap L^{1}(\RR^{n},\omega)$ for some $|x|<1$. Then using \eqref{IfromJ},
\begin{align*}
 |&I^{\epsilon}(u,x)-I(u,x)|
\leq |\tilde{J}^{\epsilon}_{z}(\delta_{2}u(1-\phi_{\epsilon}),x)\ast_{z} \eta_{\epsilon}(x) \\&-J(\delta_{2}u,x)|+\lambda (2-\sigma)\int_{|y|<\epsilon} \frac{|\delta_{2}u|}{|y|^{n+\sigma}}dy
\\&\leq |[\tilde{J}^{\epsilon}_{z}(\delta_{2}u,x)-J(\delta_{2}u,x)]\ast_{z} \eta_{\epsilon}(x)|\\&+\max\{|N^{+}(\delta_{2}u\phi_{\epsilon},x)|,|N^{-}(\delta_{2}u\phi_{\epsilon},x)|\} + \lambda (2-\sigma)\int_{|y|<\epsilon} \\&\frac{|\delta_{2}u|}{|y|^{n+\sigma}}dy
\\&\leq \sup_{|z-x|<\epsilon}|\tilde{J}^{\epsilon}_{z}(\delta_{2}u,x)-J(\delta_{2}u,x)| + C\Lambda(2-\sigma)\int_{|y|<\epsilon} \frac{|\delta_{2}u|}{|y|^{n+\sigma}}dy
\\&\leq (1+M)\rho(2\epsilon) + CM \epsilon^{2-\sigma},
\\
\end{align*}
where $M$ bounds the $C^{1,1}$ constant and $L^{1}(\RR^{n},\omega)$ norm of of $u$ at $x$, and $\rho$ is the modulus of continuity of $I$.
\end{proof}

There are two useful ways of looking at $J^{\epsilon}(u,x)$: as an elliptic operator in its own right, and as a perturbation of a fractional Laplacian by a nonlocal but zero-order quantity. In other words, for $u\in L^{1}(\RR^{n},\omega)$ define $F(u,x)$ as 
\[
 F(u,x) = \tilde{J}^{\epsilon}_{z}(\delta_{2}u(1-\phi_{\epsilon}),x)\ast_{z} \eta_{\epsilon}(x)-C_{0}(2-\sigma)\int  \frac{(1-\phi_{\epsilon})(\delta_{2}u(x,y))}{|y|^{n+\sigma}}dy,
\]
where $C_{0}$ is the proportionality constant so that
\[
 J^{\epsilon}(u,x)=-C_{0}(-\triangle)^{\sigma/2}u(x)+F(u,x).
\]

The next lemma asserts some regularity of $F(u,x)$, which derives from the fact that we regularized the equation in $x$.

\begin{lem}\label{estimates1}
 Consider measurable functions $u,v$ with $u,v\equiv g$ in $\RR^{n}\backslash B_{1}$. Then
\begin{itemize}
 \item [(i)] $\|F(u,x)-F(v,x)\|_{L^{\infty}(B_{1})}\leq C_{\epsilon}\|u-v\|_{L^{\infty}(B_{1})}$
 \item [(ii)] $|F(\tau_{h}u,x+h)-F(u,x)|\leq C_{\epsilon}|h|\|u\|_{L^{\infty}(B_{1})}$
 \item [(iii)]$F(u,x)\in C(B_{1})$ provided $u\in L^{\infty}(B_{1})$
 \item[(iv)] $F(u,x)\in C^{0,1}(B_{r})$ for $r<1$ provided $\nabla u\in L^{1}(\RR^{n},\omega)$.
\end{itemize}
\end{lem}

\begin{proof}
 For (i) apply ellipticity:
\begin{align*}
 |F(u,x)-F(v,x)| &\leq C \max \{|N^{-}((1-\phi_{\epsilon})(\delta_{2}(u-v),x)|, |N^{+}((1-\phi_{\epsilon})(\delta_{2}(u-v),x)|\}\\
&\leq C_{\epsilon}\|u-v\|_{L^{\infty}(\RR^{n})}.\\
\end{align*}
(The ellipticity estimate in the first step is done on both terms in $F$ separately.)

(ii) comes from the formula for $F$. The translation-invariant term cancels, leaving
\begin{align*}
 &|\tilde{J}^{\epsilon}_{z}(\delta_{2}\tau_{h}u(1-\phi_{\epsilon}),x+h)\ast_{z} \eta_{\epsilon}(x+h)-\tilde{J}^{\epsilon}_{z}(\delta_{2}u(1-\phi_{\epsilon}),x)\ast_{z} \eta_{\epsilon}(x)| \\ &
=|\tilde{J}^{\epsilon}_{z}(\delta_{2}u(1-\phi_{\epsilon}),x)\ast_{z} [\eta_{\epsilon}(x+2h)-\eta_{\epsilon}(x)]| \\&
\leq C_{\epsilon}|h|\|u\|_{L^{\infty}}
\end{align*}
where the first step is a change of variables and the second uses ellipticity to bound the left part of the convolution.

For (iii), use
\[
|F(u,x+h)-F(u,x)|\leq |F(\tau_{h}u,x+h)-F(u,x)|+|F(\tau_{h}u,x+h)-F(u,x+h)|.
\]
The first term is estimated by (ii), while the second, by ellipticity, is bounded by
\[
 \int_{|y|>\epsilon/2} \frac{|\delta_{2}(u-\tau_{-h}u)|}{|y|^{n+\sigma}}dy,
\]
which goes to $0$ as $h$ goes to $0$.

For (iv), take $w_{h}(x)=\frac{u(x+h)-u(x)}{|h|}$ to be the difference quotients. Then from the same computation as in (iii), we get
\[
 \frac{|F(u,x+h)-F(u,x)|}{|h|}\leq C+\int_{|y|>\epsilon/2} \frac{|\delta_{2}(w_{-h})|}{|y|^{n+\sigma}}dy.
\]
But this is bounded by assumption.
\end{proof}

Define the \emph{regularized BVP} as follows:
\[
 \begin{cases}
  I^{\epsilon}(v,x) = f_{\epsilon} & x\in B_{1} \\
  v = g_{\epsilon} & x \notin B_{1} \\
 \end{cases}
\]
where $f_{\epsilon}$ is smooth, bounded, and converges to $f$ uniformly, while $g_{\epsilon}$ is smooth, bounded, has the same modulus of continuity as $g$ on $\partial B_{1}$, and converges to $g$ in $L^{1}(\RR^{n},\omega)$. We claim this problem has classical solutions:

\begin{lem} Bounded viscosity solutions to the regularized BVP exist, are unique, and are in $C^{2,\alpha}$ for some $\alpha>0$,
\end{lem}

A simple fact we'll use in the proof is that there is always a comparison principle available between a viscosity supersolution and a classical ($C^{2}$) subsolution.

\begin{proof}
 We prove existence first. We wish to apply a fixed point argument to this equation. The appropriate Banach space will be $\mathcal{B}\equiv \{v\in C(\bar{B}_{1}):v=g_{\epsilon} \text{ on } \partial B_{1}\}$, equipped with sup norm. Define $G[u]$ to be the viscosity solution to
\[
 \begin{cases}
  -C_{0}(-\triangle)^{\sigma/2}G[u]=f_{\epsilon}-F(u,x) & x\in B_{1}\\
  G[u]=g_{\epsilon} &x\in B_{1}^{C}.
 \end{cases}
\]
Such a function exists and is unique, since the right-hand side is bounded and continuous by the estimates above. The boundary regularity estimate \ref{boundaryreg} gives that $G[u]$ is in $\mathcal{B}$, and, moreover, if $\|u\|_{L^{\infty}(B_{1})}\leq R$, $G[u]$ are equicontinuous. Part (i) of the preceding lemma shows $G$ is continuous, and it follows by Arzela-Ascoli that $G$ is compact.

Assume that $u=\kappa G[u]$ for some $\kappa\in [0,1]$. Then $u$ solves the nonlocal equation
\[
 \kappa I^{\epsilon}(u,x) - (1-\kappa)C_{0}(-\triangle)^{\sigma/2}u(x)= f_{\epsilon}
\]
in $B_{1}$ with boundary data $g_{\epsilon}$ (in the viscosity sense). Indeed, it suffices to check for all test functions $\phi$ which coincide with $u$ outside a ball of radius $<\epsilon/4$, and for these $F(u,x)=F(\phi,x)$. But this equation is uniformly elliptic with the same ellipticity constants, and so has an estimate of the form $\|u\|_{L^{\infty}}\leq C$ where $C$ depends only on $g_{\epsilon}$ and $ f_{\epsilon}$ (for instance, use comparison with a smooth barrier).

It follows from the Leray-Schauder fixed point theorem that $G$ has a fixed point $u_{\epsilon}$. By the argument just presented, this fixed point is a viscosity solution to the regularized BVP.

By applying Remark \ref{part2rem2}, we obtain that $\nabla u_{\epsilon} \in L^{1}(\RR^{n},\omega)$. But then by (iv) of Lemma \ref{estimates1}, $F(u,x)$ is Lipschitz on every ball $B_{r}$, and so applying regularity for the fractional Laplacian (see, e.g. \cite[Theorem 65]{CS2}) gives $u\in C^{2,\alpha}$.

Finally, uniqueness is now trivial using the comparison principle between classical solutions.
\end{proof}

\begin{proof}[Proof of Theorem]
 Denote by $u_{\epsilon}$ the solution to the regularized BVP. Note that the family $u_{\epsilon}$ is bounded and equicontinuous on $\bar{B}_{1}$ by boundary regularity. It thus has a uniformly convergent subsequence with a limit $u$. We have from above that $I^{\epsilon}$ converges to $I$ in norm, and the rest of the hypotheses of \cite[Lemma 5]{CS2} are satisfied. It follows that $u$ solves the BVP in the viscosity sense.
\end{proof}

\begin{rem}
 $\sigma$ was assumed to be larger than one in this section. This is because the method given here doesn't prove $C^{2,\alpha}$ regularity for solutions of $I^{\epsilon}$ when $\sigma \leq 1$.
\end{rem}

\begin{rem}\label{apriori}
 The construction here is very useful in its own right. The operator $I^{\epsilon}$ will generally retain any special structure $I$ might have possessed. For example, if $I$ is translation-invariant, so is $I^{\epsilon}$; the same is true for linearity, convexity, and ellipticity with respect to more restrictive classes. The properties discussed in Section 7 are also preserved under this regularization. Because of this, it will frequently suffice to prove estimates on $I^{\epsilon}$ which are \emph{uniform in $\epsilon$}. Say, for instance, our goal is an interior regularity estimate of the form $\|u\|_{C^{k,\alpha}(B_{1/2})}\leq C$. If such an estimate is proved for each of the classical solutions $u_{\epsilon}$ and is uniform in $\epsilon$, by compactness there is a subsequence $u_{\epsilon_{k}}$ converging to some $\bar{u}$ uniformly, where $\bar{u}$ solves the equation and satisfies the estimate $\|\bar{u}\|_{C^{k,\alpha}(B_{1/2})}\leq C$. If $k,\alpha$ are, in addition, sufficiently large to guarantee 
that $\bar{u}$ is classical (which will be the case in the applications to follow), we also recover, \emph{a posteriori}, that $\bar{u}$ is the unique solution to the Dirichlet problem.
\end{rem}

\section{Linear Schauder Theorem}

We now turn to linear integro-differential operators. Consider first the translation-invariant operator
\begin{equation}\label{linfix}
 L^{0}u(x)=(2-\sigma)\int_{\RR^{n}}\frac{\delta_{2}u(x,y)a(y)}{|y|^{n+\sigma}}dy,
\end{equation}
where $0<\lambda<a(y)<\Lambda<\infty$, but no regularity is assumed in $y$. We show this operator has interior $C^{2,\alpha}$ estimates which depend on the first derivative of the boundary data. Compare to \cite[Theorem 4.1]{CS3} and \cite[Corollary 1.7]{DK}.
\begin{thm}\label{linthm1}
Assume $\sigma>\sigma_{0}>1$ and $L^{0}u=f$ on $B_{1}$. If $f\in C^{0,1}(B_{1})$ and $u\in C^{0,1}(B_{2}\backslash B_{1})\cap \cA^{1}$, then $u\in C^{2,\alpha}(B_{1/2})$ and
\begin{equation}
 \|u\|_{C^{2,\alpha}(B_{1/2})}\leq C \left(\|f\|_{C^{0,1}(B_{1})} + \|u\|_{C^{0,1}(B_{2}\backslash B_{1})} + \|u\|_{L^{1}(\RR^{n},\omega)} + [u]_{\cA^{1}}+\|u\|_{L^{\infty}(B_{1})}  \right)
\end{equation}
where $C$, $\alpha$ are universal. 
\end{thm}

\begin{proof}
 We prove the a priori estimate; in other words, we assume $u$ solves the equation classically. The theorem then holds for viscosity solutions by using Remark \ref{apriori}.

Theorem \ref{part1} applies to $u$ to give
\[
 \|u\|_{C^{1}(B_{1/2})}\leq C
\]
and
\[
  |w_{h}|\leq C (1-|x|)^{s-1},
\]
where $w_{h}$ is the difference quotient $(u(x+h)-u(x))/|h|$ and $|h|<(1-|x|)/2$. The difference quotients satisfy the following when $|x|<1/2$ and $|h'|<1/4$:
\[
 L^{0}w_{h'}(x)=\frac{f(x+h')-f(x)}{|h'|},
\]
which is bounded by $[f]_{C^{0,1}}$. We now check that $w_{h}\in L^{1}(\RR,\omega)$:
\begin{align*}
 \int \frac{|u(y+h)-u(y)|}{|h||y|^{n+\sigma_{0}}}dy
&\leq \int_{B_{1+2|h|}^{C}}+\int_{B_{1+2|h|}\backslash B_{1-2|h|}}+\int_{B_{1-2|h|}} \\
&\leq [u]_{\cA^{1}}+C\|u\|_{L^{\infty}(B_{2})} + C \int_{B_{1}}(1-|y|)^{s-1}dy,\\
\end{align*}
and the last term is integrable. Applying Theorem \ref{part2} to $w_{h}$ on $B_{1/2}$ gives $\|w_{h}\|_{C^{1,\alpha}(B_{1/4}}\leq C$ uniformly in $h$, which after a rescaling gives the conclusion desired.
\end{proof}

Next consider the variable-coefficient case:
\begin{equation}\label{linvar}
 Lu(x)=(2-\sigma)\int_{\RR^{n}}\frac{\delta_{2}u(x,y)a(x,y)}{|y|^{n+\sigma}}dy,
\end{equation}
where $0<\lambda<a(x,y)<\Lambda<\infty$ uniformly in $y$. Associated to $L$ is a family of operators with fixed coefficients
\begin{equation}
 L^{x_{0}}u(x)=(2-\sigma)\int_{\RR^{n}}\frac{\delta_{2}u(x,y)a(x_{0},y)}{|y|^{n+\sigma}}dy,
\end{equation}
each of which are of the form \eqref{linfix}. A computation in \cite[Theorem 61]{CS2} shows that
\begin{equation} \label{linhyp1}
|a(x,y)-a(0,y)|<\eta \ \text{for} \ x\in B_{1}, y\in \RR^{n}\backslash{0}, 
\end{equation}
then $\|L-L^{0}\|_{\sigma} \leq C\eta$. 

The objective is to prove a Schauder-type theorem for such operators without assuming regularity in $y$:

\begin{thm}\label{linschauder}
 Assume $Lu=f$ on $B_{1}$ in the viscosity sense, where $L$ is as above and  $\sigma>\sigma_{0}>1$. Fix $1>\alpha^{*}>0$, and assume further that 
\begin{itemize}
 \item[(A)] $\|u\|_{C^{0,\alpha^{*}}(B_{1}^{C})}\leq 1$,
 \item[(B)]$\|f\|_{C^{0,\alpha^{*}}(B_{1})}\leq 1$, and
 \item[(C)] $|a(x,y)-a(x_{0},y)|\leq |x-x_{0}|^{\alpha^{*}}$ for $x,x_{0}\in B_{3/4}$ and $y\in \RR^{n}$.
\end{itemize}
Then for every $\alpha<\alpha^{*}$,
\[
 \|u\|_{ C^{\sigma+\alpha}(B_{1/2})}\leq C
\]
with $C$ depending only on $n,\lambda, \Lambda, \sigma_{0},\alpha^{*}$, and $\alpha$.
\end{thm}

Our proof follows the argument in \cite{BFV}. The major difference is the need to gain $\sigma$ derivatives of regularity rather than just one. As the gain in regularity comes from Theorem \ref{linthm1} for the constant-coefficient equations, we need to prove a perturbative estimate along the lines of Theorem \ref{part2}. Notice that such an estimate is inherently limited to proving $C^{\sigma-\kappa}$ regularity; anything higher will result in the sequence of rescaled functions no longer having integrable tails. As the proof of this Lemma is virtually identical to that of Theorem \ref{part2}, we emphasize only the differences.

\begin{lem}\label{linlem1}
Assume $Lu=f$ on $B_{1}$ in the viscosity sense and  $\sigma>\sigma_{0}>1$. Then for every $\sigma_{0}-1>\kappa>0$ there is an $\eta>0$ such that if 
\begin{itemize}
 \item $\|u\|_{C^{0,1}(\RR^{n})}\leq 1$,
 \item$\|f\|_{L^{\infty}(B_{1})}\leq \eta$, and
 \item $|a(x,y)-a(x_{0},y)|\leq \eta $ for $x,x_{0}\in B_{1}$ and $y\in \RR^{n}$,
\end{itemize}
then $u \in C^{\sigma-\kappa}(B_{1/2})$ and
\[
 \|u\|_{  C^{\sigma-\kappa}(B_{1/2})  }\leq C
\]
where $C$ depends only on $n,\lambda, \Lambda, \sigma_{0},$ and $\kappa$.
\end{lem}

\begin{proof}
We use the same notation as the proof of Theorem \ref{part2}. In particular, we find a sequence $l_{k}=a_{k}+b_{k}\cdot x$ such that 
\begin{itemize}
 \item[(i.k)] $\sup_{B_{\gamma^{k}}}|u(x)-l_{k}(x)|\leq \gamma^{k(\sigma-\kappa)}$
 \item[(ii.k)] $|u(\gamma^{k}x)-l_{k}(\gamma^{k}x)|\leq \gamma^{k(\sigma-\kappa)}(1+|x|^{\sigma-\kappa/2})$
 \item[(iii.k)] $|u(\gamma^{k}(x+h))-u(\gamma^{k}(x))|\leq |h|\gamma^{k(\sigma-\kappa)}(1+|2x|^{\sigma-\kappa/2})$ for $|h|<1$
 \item[(iv.k)] $|a_{k+1}-a_{k}|+\gamma^{k}|b_{k+1}-b_{k}|\leq C_{1}\gamma^{k(\sigma-\kappa)}$.
\end{itemize}
As before construct the rescaled function $w_{k}$. Then $w_{k}$ satisfies the rescaled equation $L_{k}w_{k}=f_{k}$ where the right-hand side is bounded by $\eta$. Then let $v_{k}$ solve
\[
\begin{cases}
L^{0}_{k}v_{k}=0 & x\in B_{1} \\
v_{k}=w_{k} & x \notin B_{1}.
\end{cases}
\]
Using the assumptions and \eqref{linhyp1}, $\|L^{0}_{k}-L_{k}\|\leq C\eta$. Choose $\eta$ sufficiently small so that, combining with (ii.k) and (iii.k), \cite[Lemma 7]{CS2} implies $|w_{k}-v_{k}|\leq \gamma^{\sigma-\kappa/3}$. Also using (ii.k) and (iii.k), Theorem \ref{linthm1} applies to $v_{k}$ to give $\|v_{k}\|_{C^{2,\alpha}(B_{1/2})}\leq C_{0}$.

(i.k+1)-(iv.k+1) can now be checked the same way as before. Applying to translates of the operator and scaling completes the argument. 
\end{proof}

\begin{proof}[Proof of Theorem]

First, it suffices to work with classical solutions, using Remark \ref{apriori} (that the equation satisfies \eqref{strongcont} follows from (C)).

We will prove the estimate
\begin{equation} \label{linint1}
\|u\|_{C^{\sigma+\alpha}(B_{1/100})}\leq C_{\delta}+\delta \|u\|_{C^{\sigma+\alpha}(B_{1/2})}.
\end{equation}
As the hypotheses (A)-(C) only decrease when the equation is dilated, the same is true at every scale:
\[
\|u\|^{*}_{C^{\sigma+\alpha}(B_{r/100}(x))}\leq C_{\delta}+\delta \|u\|^{*}_{C^{\sigma+\alpha}(B_{r}(x))}.
\]
where $B_{r}(x)\subset B_{1/2}$ and $\|\cdot\|^{*}$ are the adimentional \Holder norms as in \cite{BFV}. It then follows from \cite[Lemma 8]{BFV} that $\|u\|_{C^{\sigma+\alpha}(B_{1/4})}\leq C$.

Now consider the equation satisfied by $q_{h}(x)=[u(x+h)-u(x)]/|h|^{\alpha^{*}}$:
\[
Lq_{h}(x)=\frac{f(x+h)-f(x)}{|h|^{\alpha^{*}}}-(2-\sigma)\int \frac{\delta_{2}u(x,y)[a(x+h)-a(x)]}{|h|^{\alpha^{*}}|y|^{n+\sigma}}dy
\]
for $x \in B_{1/4}$. The first term on the right is bounded by assumption (B), while the second, using (C), is bounded by $C(1+\|u\|_{C^{\sigma + \epsilon}(B_{1/2})})$. $q_{h}$ itself is in $L^{1}(\RR^{n},\omega)$ uniformly in $h$ by the usual argument with the weighted \Holder estimate from Remark \ref{part2rem2}. Thus Theorem \ref{part2} applies to give $q_{h}\in C^{1,\alpha}(B_{1/8})$.

We now make the usual reductions to apply Lemma \ref{linlem1}. In particular, let $\phi$ be a smooth cutoff supported on $B_{1/8}$ and $\equiv 1$ on $B_{1/16}$. Then $|L\phi q_{h}(x)|\leq C(1+\|u\|_{C^{\sigma + \epsilon}(B_{1/2})})$ for $|x|<1/32$ and $\phi q_{h}$ is $C^{0,1}(\RR^{n})$. Set $\kappa = \alpha^{*}-\alpha$ and $\epsilon = \alpha/2$. Then Lemma \ref{linlem1} gives
\[
\|q_{k}\|_{C^{\sigma -(\alpha^{*}-\alpha)}(B_{1/100})}\leq C(1+\|u\|_{C^{\sigma + \alpha/2}(B_{1/2})}).
\]
It follows that unless $\sigma+\alpha=2$
\[
\|u\|_{C^{\sigma +\alpha}(B_{1/100})}\leq C(1+\|u\|_{C^{\sigma + \alpha/2}(B_{1/2})})\leq C_{\delta}+\delta\|u\|_{C^{\sigma + \alpha}(B_{1/2})}
\]
with the last step by interpolation. This proves \eqref{linint1}. If $\sigma+\alpha=2$, simply prove the theorem for a slightly larger $\alpha'$ instead.
\end{proof}

Using the linear structure, we may further weaken the hypotheses:

\begin{cor} \label{lincor}
 The conclusion of Theorem \ref{linschauder} still holds under the alternative hypothesis
\begin{itemize}
 \item[(A')] $\|u\|_{L^{1}(\RR^{n},\omega}+[u]_{\cA^{\alpha^{*}}}+\|u\|_{C^{\alpha^{*}}(B_{2}\backslash B_{1})}\leq 1$.
\end{itemize}
\end{cor}

\begin{proof}
 Let $\phi$ be a smooth cutoff supported on $B_{2}$ and $\equiv 1$ on $B_{3/2}$. Then $\phi u$ satisfies (A). $L\phi u= f - L(1-\phi)u$ by linearity, and for $|x|<1$
\begin{align*}
 &|L(1-\phi)u(x+h)-L(1-\phi)u(x)| \\
&\leq C\int_{B_{3/2-|h|}^{C}}\frac{|(1-\phi)u(y+h)-(1-\phi)u(y)|}{|y-x|^{n+\sigma^{0}}}dy \leq C|h|^{\alpha^{*}},\\
\end{align*}
so $[L \phi u]_{C^{\alpha^{*}}(B_{1})}\leq C$. A similar argument reveals $|L\phi u|\leq C$. This means $\phi u$ satisfies (B) as well, and Theorem \ref{linschauder} applies.
\end{proof}

\begin{rem} We restricted the values of $\alpha^{*}$ to $(0,1)$ above to simplify the exposition. A modification of the proof will go through whenever $\delta_{k}u(x,h)/|h|^{\alpha^{*}}$ are in $L^{1}(\RR^{n},\omega)$ uniformly in $h$, where $\delta_{k}$ are the $k$-th symmetric differences (for some $k>\alpha^{*}$). In general, this will hold for $\alpha^{*}\in (0,1+s)$, where $s$ is the known boundary regularity of the equation. 
\end{rem}

\section{Nonlinear Applications}

Let $I$ be a nonlinear nonlocal operator of the form
\begin{equation}\label{nonlinop}
 I(u,x)=(2-\sigma)\int_{\RR^{n}}\frac{\rho(\delta_{2}u,y)}{|y|^{n+\sigma}}dy.
\end{equation}
We assume that $\rho(0)=0$ and for each $y$, $\rho(z,y)$ is twice continuously differentiable in $z$, and furthermore satisfies
\begin{equation}\label{nonlocassump1}
 0<\lambda<\partial_{z}\rho(z,y)<\Lambda<\infty \qquad \forall (z,y)\in \RR \times \RR^{n}\backslash\{0\}
\end{equation}
and
\begin{equation}\label{nonlocassump2}
 |\partial^{2}_{z}\rho(z,y)|\leq C_{0}(1+|y|^{-2}) \qquad \forall (z,y)\in \RR \times \RR^{n}\backslash\{0\}.
\end{equation}
Note that \eqref{nonlocassump1} guarantees uniform ellipticity with respect to $\cL_{0}$. The examples we have in mind are $\rho(z,y)=\bar{\rho}(z)$, which was mentioned in the introduction, and more generally $\rho(z,y)=|y|^{p}\bar{\rho}(z|y|^{-p})$ with $p\in [0,2]$. $\bar{\rho}$ is assumed $C^{2}$ with $\lambda<\bar{\rho}'<\Lambda$ and $|\bar{\rho}''|\leq C_{0}$. With $p=2$, this is the equation mentioned in \cite[Section 6]{CS1}.

Assume $u$ solves the BVP \eqref{BVP} for such an operator with $f,g$ assumed smooth ($C^{1,\alpha}$ suffices). We will show $u\in C^{1+\sigma+\alpha}(B_{1/2})$. First, a technical lemma to show we can work with classical solutions:

\begin{lem}\label{nonlinlem1}
 Assume $\sigma>\sigma_{0}>1$, $u$ solves the BVP \eqref{BVP}, $f,g$ are smooth, and $I$ is as in \eqref{nonlinop} with $\rho$ satisfying \eqref{nonlocassump1} and \eqref{nonlocassump2}. Then there is a sequence of operators $I^{\epsilon}$ of the same type, with $\rho^{\epsilon}$ satisfying the same assumptions, and functions $u_{\epsilon}$ solving
\[
 \begin{cases}
  I^{\epsilon}(u_{\epsilon},x)=f(x) & x \in B_{1}\\
  u_{\epsilon}=g & x\notin B_{1}
 \end{cases}
\]
such that $u_{\epsilon}$ is $C^{2,\alpha}$ on the interior of $B_{1}$ and $u_{\epsilon}\rightrightarrows u$ uniformly on $B_{1}$.
\end{lem}

\begin{proof}
Simply observe that when we apply the construction used in the proof of Theorem \ref{ExUnqThm}, $I^{\epsilon}$ is of the form \eqref{nonlinop}, with $\rho^{\epsilon}$ given by
\[
\rho^{\epsilon}(z,y)=\rho(z(1-\phi_{\epsilon}(y)),y)+\phi_{\epsilon}(y)\lambda z. 
\]
A computation shows this satisfies \eqref{nonlocassump1} and \eqref{nonlocassump2}.
\end{proof}

\begin{thm}\label{nonlinthm}
 Let $I$ be as above, $\sigma>\sigma_{0}>1$, and $u$ satisfy $Iu=f$ in $B_{2}$, $u=g$ in $B_{2}^{C}$ with $f,g$ smooth. Then $u\in C^{1+\sigma+\alpha}(B_{1/2}).$
\end{thm}

\begin{proof}
By Lemma \ref{nonlinlem1}, it suffices to prove this for classical solutions.

 Let $w_{h}(x)=\frac{u(x+h)-u(x)}{|h|}$ be the difference quotients. They satisfy the following equation for $x\in B_{1}$ and $|h|<1/2$:
\begin{align*}
 \frac{f(x+h)-f(x)}{|h|}&=
\frac{2-\sigma}{|h|}\int\frac{\rho(\delta_{2}u(x+h,y),y)-\rho(\delta_{2}u(x,y),y)}{|y|^{n+\sigma}}dy \\
&=\frac{2-\sigma}{|h|}\int\int_{0}^{1}\frac{\partial_{t}\rho(t\delta_{2}u(x+h,y)+(1-t)t\delta_{2}u(x,y),y)}{|y|^{n+\sigma}}dtdy \\
&=\int\frac{a(x,y)\delta_{2}w_{h}(x,y)}{|y|^{n+\sigma}}dy\\
\end{align*}
where
\[
 a(x,y)=\int_{0}^{1}\partial_{z}\rho(t\delta_{2}u(x+h,y)+(1-t)\delta_{2}u(x,y),y)dt.
\]
This is a linear elliptic equation with a \Holder right-hand side. Take $s$ as in Theorem \ref{boundaryreg} applied to $u$; then $u\in C^{0,s}(\RR^{n})$. It follows from Corollary \ref{part1cor} that $w_{h}\in L^{1}(\RR^{n},\omega)\cap \cA^{s/2}$. Theorem \ref{part2} applies to give $w_{h}\in C^{1.\alpha}(B_{1/2})$, so $u\in C^{2,\alpha}(B_{1/2})$.

We claim $|a(x,y)-a(x',y)|<C|x-x'|^{s}$ uniformly in $y$ for $x,x'\in B_{1/4}$. Indeed, by the fundamental theorem of calculus,
\begin{align*}
& a(x,y)-a(x',y)\\
&=\int_{0}^{1}\int_{0}^{1}\partial^{2}_{z}\rho(s[t\delta_{2}u(x+h,y)+(1-t)\delta_{2}u(x,y)]+(1-s)[t\delta_{2}u(x'+h,y)+(1-t)\delta_{2}u(x',y)],y)\\
&[t\delta_{2}(u(x+h,y)-u(x'+h,y))+(1-t)\delta_{2}(u(x,y)-u(x',y))]dsdt \\
&\leq \int_{0}^{1}C_{0}(1+|y|^{-2})[t\delta_{2}(u(x+h,y)-u(x'+h,y))+(1-t)\delta_{2}(u(x,y)-u(x',y))]\\
\end{align*}
using \eqref{nonlocassump2}. Then use that for $|h|,|y|<1/8$, 
\[
\frac{|\delta_{2}(u(x+h,y)-u(x'+h,y))|}{|y|^{2}}+\frac{|\delta_{2}(u(x,y)-u(x',y))|}{|y|^{2}}\leq C\|u\|_{C^{2,\alpha}}|x-x'|^{\alpha} ,
\]
while for $|y|>1/8$, $u\in C^{0,s}(\RR^{n})$ and the constant in front is bounded.

The conclusion now follows from applying Corollary \ref{lincor} and rescaling.
\end{proof}

\section*{Acknowledgements}

The author would like to thank his PhD supervisor, Luis Caffarelli, for many valuable conversations regarding this project, especially concerning the use of boundary regularity ideas. He is also grateful for the errors and suggestions pointed out by Hector Chang Lara and the referees. This work was supported by National Science Foundation grant NSF DMS-1065926.

\bibliographystyle{plain}
\bibliography{RoughKernels0428}

\end{document}